\documentclass[a4paper,10pt]{amsart}
\usepackage{mathrsfs}
\usepackage{cases}
\usepackage{amsfonts}
\usepackage{graphicx}
\usepackage{amsmath}
\usepackage{amssymb}
\usepackage[all]{xypic}
\usepackage[all]{xy}
\usepackage{color}
\usepackage{colordvi}
\usepackage{multicol}
\begin{document}
\input xy
\xyoption{all}

\renewcommand{\mod}{\operatorname{mod}\nolimits}
\newcommand{\ind}{\operatorname{inj.dim}\nolimits}
\newcommand{\id}{\operatorname{id}\nolimits}
\newcommand{\Mod}{\operatorname{Mod}\nolimits}
\newcommand{\End}{\operatorname{End}\nolimits}
\newcommand{\rep}{\operatorname{rep}\nolimits}
\newcommand{\Ext}{\operatorname{Ext}\nolimits}
\newcommand{\Tor}{\operatorname{Tor}\nolimits}
\newcommand{\Hom}{\operatorname{Hom}\nolimits}
\newcommand{\aut}{\operatorname{Aut}\nolimits}
\renewcommand{\Im}{\operatorname{Im}\nolimits}
\newcommand{\Ker}{\operatorname{Ker}\nolimits}
\newcommand{\Iso}{\operatorname{Iso}\nolimits}
\newcommand{\coker}{\operatorname{Coker}\nolimits}
\newcommand{\Cone}{{\operatorname{Cone}\nolimits}}
\renewcommand{\Vec}{{\operatorname{Vec}\nolimits}}
\newcommand{\pd}{\operatorname{proj.dim}\nolimits}
\newcommand{\add}{\operatorname{add}\nolimits}
\newcommand{\pr}{\operatorname{pr}\nolimits}
\newcommand{\cc}{{\mathcal C}}
\newcommand{\ce}{{\mathcal E}}
\newcommand{\cs}{{\mathcal S}}
\newcommand{\cf}{{\mathcal F}}
\newcommand{\cx}{{\mathcal X}}
\newcommand{\cy}{{\mathcal Y}}
\newcommand{\cl}{{\mathcal L}}
\newcommand{\ct}{{\mathcal T}}
\newcommand{\cu}{{\mathcal U}}
\newcommand{\cm}{{\mathcal M}}
\newcommand{\cv}{{\mathcal V}}
\newcommand{\ch}{{\mathcal H}}
\newcommand{\ca}{{\mathcal A}}
\newcommand{\mcr}{{\mathcal R}}
\newcommand{\cb}{{\mathcal B}}
\newcommand{\ci}{{\mathcal I}}
\newcommand{\cj}{{\mathcal J}}
\newcommand{\cp}{{\mathcal P}}
\newcommand{\cg}{{\mathcal G}}
\newcommand{\cw}{{\mathcal W}}
\newcommand{\co}{{\mathcal O}}
\newcommand{\cd}{{\mathcal D}}
\newcommand{\cn}{{\mathcal N}}
\newcommand{\ck}{{\mathcal K}}
\newcommand{\calr}{{\mathcal R}}
\newcommand{\ol}{\overline}
\newcommand{\ul}{\underline}
\newcommand{\cz}{{\mathcal Z}}
\newcommand{\st}{[1]}
\newcommand{\ow}{\widetilde}
\renewcommand{\P}{\mathbf{P}}
\newcommand{\pic}{\operatorname{Pic}\nolimits}
\newcommand{\Spec}{\operatorname{Spec}\nolimits}
\newtheorem{theorem}{Theorem}[section]
\newtheorem{acknowledgement}[theorem]{Acknowledgement}
\newtheorem{algorithm}[theorem]{Algorithm}
\newtheorem{axiom}[theorem]{Axiom}
\newtheorem{case}[theorem]{Case}
\newtheorem{claim}[theorem]{Claim}
\newtheorem{conclusion}[theorem]{Conclusion}
\newtheorem{condition}[theorem]{Condition}
\newtheorem{conjecture}[theorem]{Conjecture}
\newtheorem{construction}[theorem]{Construction}
\newtheorem{corollary}[theorem]{Corollary}
\newtheorem{criterion}[theorem]{Criterion}
\newtheorem{definition}[theorem]{Definition}
\newtheorem{example}[theorem]{Example}
\newtheorem{exercise}[theorem]{Exercise}
\newtheorem{lemma}[theorem]{Lemma}
\newtheorem{notation}[theorem]{Notation}
\newtheorem{problem}[theorem]{Problem}
\newtheorem{proposition}[theorem]{Proposition}
\newtheorem{remark}[theorem]{Remark}
\newtheorem{solution}[theorem]{Solution}
\newtheorem{summary}[theorem]{Summary}
\newtheorem*{thm}{Theorem}

\makeatletter
\@addtoreset{equation}{section}
\makeatother
\renewcommand{\theequation}{\arabic{section}.\arabic{equation}}

\def \bp{{\mathbf p}}
\def \bA{{\mathbf A}}
\def \bL{{\mathbf L}}
\def \bF{{\mathbf F}}
\def \bS{{\mathbf S}}
\def \bC{{\mathbf C}}
\def \bD{{\mathbf D}}
\def \Z{{\Bbb Z}}
\def \F{{\Bbb F}}
\def \C{{\Bbb C}}
\def \N{{\Bbb N}}
\def \Q{{\Bbb Q}}
\def \G{{\Bbb G}}
\def \X{{\Bbb X}}
\def \P{{\Bbb P}}
\def \K{{\Bbb K}}
\def \E{{\Bbb E}}
\def \A{{\Bbb A}}
\def \BH{{\Bbb H}}
\def \T{{\Bbb T}}

\title[Modified Ringel-Hall algebras]{Modified Ringel-Hall algebras, naive lattice algebras and lattice algebras}
\author[Lin]{Lin Ji}
\address{Department of Mathematics, Sichuan University, Chengdu 610064, P.R.China and
Department of Mathematics and Statistics, Fuyang Normal University, Fuyang 236037, P.R.China}
\email{jlin@fync.edu.cn}

\thanks{This work was supported partially by the National Natural Science Foundations of China (Grant No. 11701473) and Youth Talent Foundation of Fuyang Normal University (Grant No. rcxm201803)}
\begin{abstract}
For a given hereditary abelian category satisfying some finiteness conditions, in certain twisted cases it is shown that the modified Ringel-Hall algebra is isomorphic to the naive lattice algebra and there exists an epimorphism from the modified Ringel-Hall algebra to the lattice algebra. Furthermore, the kernel of this epimorphism is described explicitly. Finally, we show that the naive lattice algebra is invariant under the derived equivalences of hereditary abelian categories.
\end{abstract}

\maketitle
\section{Introduction}\label{section introduction}

The Ringel-Hall algebra $\ch(\ca)$ associated to an abelian category $\ca$ satisfying some finiteness conditions is a vector space over $\mathbb{Q}$ with basis parameterized by the isomorphism classes of objects in the category $\ca$. And the positive part of the quantum enveloping algebra can be realized by the Ringel-Hall algebras (see \cite{Gr,Lu,R2}). Since then much of the work was concentrated on the realization of the whole quantum group, see for example \cite{P1,PX1,PX2,SV,R3,Hub}.

Xiao~\cite{X} realized the whole quantum group via the reduced Drinfeld double of a Hall algebra by piecing together two Borel parts. Furthermore, Cramer~\cite{Cr} proved that the  Drinfeld double Hall algebra is invariant under the derived equivalences. Recently, Bridgeland \cite{Br} considered the Hall algebras over the category of $\mathbb{Z}/2$-graded complexes with projective components  which  provide  natural realizations of the whole quantum groups. Lu and Peng~\cite{LuP} introduced the modified Ringel-Hall algebra $\cm\ch_{\mathbb{Z}/2}(\ca)$ of $\ca$ over the category of $\mathbb{Z}/2$-graded complexes, where $\ca$ is a hereditary abelian category which may not have enough projectives. Moreover,  it is proved that the componentwise twisted modified Ringel-Hall algebra $\cm\ch_{\mathbb{Z}/2,tw}(\ca)$ is isomorphic to the Drinfeld double Hall algebra of $\ch_{tw}^{e}(\ca)$ which generalizes the construction of Bridgeland \cite{Br} and Gorsky \cite{Gor13}.  In our previous work \cite{LinP}, we extended the construction of~\cite{LuP} to define the modified Ringel-Hall algebra $\cm\ch(\ca)$ for the category of bounded complexes $\cc^b(\ca)$ and showed that the twisted modified Ringel-Hall algebra is invariant under derived equivalences.

For a finitary hereditary abelian category $\mathcal{A}$, Kapranov~\cite{Kap} introduced the {\it lattice algebra} $\mathcal{L}(\ca)$ and the {\it naive lattice algebra} $\mathcal{N}(\ca)$ for $\ca$ and proved that the lattice algebra $\mathcal{L}(\mathcal{A})$ is invariant under derived equivalences of hereditary abelian categories.
On the other hand, one may also have the so-called {\it derived Hall algebra} $\mathcal{DH}(\mathcal{A})$ of $\mathcal{A}$ in the sense of ~\cite{T, XX1}.
The relation between lattice algebra and derived Hall algebra has been investigated by Sheng and Xu in~\cite{SX}.
For the hereditary abelian category $\ca$ with enough projectives,  Zhang~\cite{Z} introduced the $m$($m=0$ or $m>2$)-lattice algebra and proved that it is isomorphic to the Bridgeland's Hall algebra of $m$-cyclic projective complexes of $\ca$. In particular, the so-called $0$-lattice algebra is the naive lattice algebra.

In the present paper, we mainly study the relations among the modified Ringel-Hall algebra $\cm\ch(\ca)$, the naive lattice algebra $\cn(\ca)$ and the lattice algebra $\cl(\ca)$ for a hereditary abelian category $\ca$ which may not have enough projectives but satisfies the following {\it finiteness conditions} over a finite field $k$:
\begin{itemize}
\item [(1)] $\ca$ is an essentially small $k$-linear category,\\
\item [(2)] $\ca$ is finitary, {\it i.e.} $\dim_k\Hom_{\ca}(M,N)<\infty$ and $\dim_{k}\Ext^1_{\ca}(M,N)<\infty$ for any $M, N\in\ca$, \\
\end{itemize}
In particular, we give the realizations of the naive lattice algebra and the lattice algebra via the modified Ringel-Hall algebra respectively. As an application, we obtain the derived invariance of the naive lattice algebras. Specifically, we have the following main results of this paper.

\begin{theorem}[Theorem \ref{theorem of isomorphism of naive lattice algebras and modified hall algebras}]\label{modified and naive}
The naive lattice algebra $\cn(\ca)$ is isomorphic to the componentwise twisted modified Ringel-Hall algebra $\cm\ch_{ctw}(\ca)$.
\end{theorem}

\begin{theorem}[Theorem \ref{theorem epimorphism from modified to lattice}]\label{modifed and lattice}
Let $\cm\ch_{rtw}(\ca)$ be the relative twisted modified Ringel-Hall algebra and $\cl_*(\ca)$ be the Drinfeld dual lattice algebra which is isomorphic to $\cl(\ca)$. There is an epimorphism $\varphi:\cm\ch_{rtw}(\ca)\rightarrow\cl_*(\ca)$, given by $$K_{\alpha, n}\mapsto K^{(-1)^n}_{\alpha}, U_{A,n}\mapsto Z^{(n)}_{A}$$ for any $A\in\Iso(\ca)$, $n\in\mathbb{Z}$ and $\alpha\in K_0(\ca)$. Moreover, $\cl_*(\ca)\cong \cm\ch_{rtw}(\ca)/I$, where $I$ is the ideal of $\cm\ch_{rtw}(\ca)$ generated by the set $$\{K_{\alpha, n+1}K_{\beta, n}-K_{\alpha-\beta, n+1}|\alpha, \beta\in K_{0}(\ca), n\in\mathbb{Z}\}.$$
\end{theorem}

\begin{theorem}[Theorem \ref{theorem invariance of naive and lattice}]\label{main invariance in introduction}
Let $\cb$ be  a hereditary abelian $k$-category satisfying the finiteness conditions $(1)$-$(2)$. If there exists a derived equivalence $F: D^b(\ca)\rightarrow D^b(\cb)$, then we have the following  isomorphism of algebras
$$\cn(\ca)\cong \cn(\cb).$$
\end{theorem}

The paper is organized as follows. In Section \ref{section modified}, we recall the definition and the structure of the modified Ringel-Hall algebras. In Section \ref{section modified and naive}, we prove Theorem \ref{modified and naive}. Section \ref{section modified and lattice} is devoted to investigating the relation between the modified Ringel-Hall algebra and the lattice algebra, and Theorem \ref{modifed and lattice} is proved in this section. In Section \ref{the derived invariance}, we show Theorem \ref{main invariance in introduction} by the same method used in \cite{LinP}.

Unless otherwise specified, throughout this paper $k$ denotes a finite field with $q$ elements and put $v=\sqrt{q}$, $\ca$ is a hereditary abelian $k$-category satisfying the finiteness conditions (1)-(2).

\vspace{0.2cm} \noindent{\bf Acknowledgments.}
The author is deeply indebted to Professor Liangang Peng for his beneficial discussions. The author is also grateful to Changjian Fu and Ming Lu for their inspirations.

\section{Modified Ringel-Hall Algebras}\label{section modified}
\subsection{Ringel-Hall Algebras}
Let $\varepsilon$ be an essentially small exact category, linear over the finite field $k$. The set of the isomorphism classes of $\varepsilon$ is denoted by $\Iso(\varepsilon)$, and $\widehat{A}$ denotes the corresponding element in the Grothendieck group $K_0(\varepsilon)$ for an object $A\in\varepsilon$. Assume that $\varepsilon$ has finite morphism and extension spaces:
$$|\Hom_\varepsilon(A,B)|<\infty,\quad |\Ext^1_{\varepsilon}(A,B)|<\infty,\,\,\forall A,B\in\varepsilon.$$

Given objects $A,B,C\in\varepsilon$, define $\Ext_\varepsilon^1(A,C)_B\subseteq \Ext_\varepsilon^1(A,C)$ to be the subset parameterizing extensions whose middle term is isomorphic to $B$. The {\it Ringel-Hall algebra} $\ch(\varepsilon)$ is the $\Q$-vector space $\bigoplus_{[A]\in \Iso(\varepsilon)}\mathbb{Q}[A]$ with basis parametrized by the isomorphism classes of objects endowed with the multiplication
$$[A]\diamond [C]=\sum_{[B]\in \Iso(\varepsilon)}\frac{|\Ext^1_{\varepsilon}(A,C)_B|}{|\Hom_\varepsilon(A,C)|}[B].$$
It is well-known that
the algebra $\ch(\varepsilon)$ is associative and unital  with unit $[0]$, where $0$ is the zero object of $\varepsilon$, see \cite{R1} and also \cite{R2,P1,Hub,Br}.

In this section, we fix a hereditary abelian $k$-category $\ca$ satisfying the finiteness conditions (1)-(2)({\it cf.} Section~\ref{section introduction}). For any objects $A, B\in\mathcal{A}$, we define the Euler form
$$\langle A, B\rangle=\prod _{i\geq 0}|\Ext^i_{\ca}(A,B)|^{(-1)^i}.$$
It is easy to check that this form descends to a bilinear form on the Grothendieck group $K_0(\ca)$ of $\ca$ which we denote it by the same symbol:
$$\langle -, -\rangle: K_0(\ca)\times K_0(\ca)\rightarrow\mathbb{Q}^{\times}.$$
Denote by $(\widehat{A}, \widehat{B})=\langle \widehat{A}, \widehat{B}\rangle\langle \widehat{B}, \widehat{A}\rangle$ the symmetrized Euler form.

\begin{definition}[\cite{R2,Gr}]\label{definition of twisted Hall algebra}
$(1)$ The twisted Ringel-Hall algebra $\ch_{tw}(\ca)$ is the $\mathbb{Q}(v)$-vector space with the same basis as $\ch(\ca)$, and the twisted multiplication is defined by
$$[A]* [B]=\sqrt{\langle \widehat{A},\widehat{B}\rangle}[A]\diamond [B]$$
for any $[A],[B]\in\Iso(\ca)$.

$(2)$ The extended twisted Ringel-Hall algebra $\ch^e_{tw}(\ca)$ is defined as an extension of $\ch_{tw}(\ca)$ by adjoining symbols $k_\alpha$ for classes $\alpha\in K_0(\ca)$, and imposing relations
$$k_\alpha* k_\beta=k_{\alpha+\beta},\quad k_\alpha* [B]=\sqrt{( \alpha,\widehat{B})}  [B]* k_\alpha$$
for $\alpha,\beta\in K_0(\ca)$ and $[B]\in\Iso(\ca)$. Note that $\ch^e_{tw}(\ca)$ is an associative algebra over $\mathbb{Q}(v)$ with a basis consisting of the elements $[B]*k_\alpha$ for $[B]\in\Iso(\ca)$ and $\alpha\in K_0(\ca)$.
\end{definition}

\subsection{Modified Ringel-Hall algebras}
Let $\ch(\cc^b(\ca))$ be the Ringel-Hall algebra of $\cc^b(\ca)$, {\it i.e.} for any $L,M\in \cc^b(\ca)$, the Hall product is defined to be the following sum:
\begin{equation*}
[L]\diamond [M]=\sum_{[X]\in \Iso(\cc^b(\ca))}\frac{|\Ext^1_{\cc^b(\ca)}(L,M)_X|}{|\Hom_{\cc^b(\ca)}(L,M)|}[X].
\end{equation*}

Let $\ch(\cc^b(\ca))/I$ be the quotient algebra, where $I$ is the  ideal of $\ch(\cc^b(\ca))$ generated by all differences $[L]-[K\oplus M]$, if there is a short exact sequence $K\rightarrowtail L\twoheadrightarrow M$ in $\cc^b(\ca)$ with $K$ acyclic. We also denote by $\diamond$ the induced multiplication in $\ch(\cc^b(\ca))/I$.

We denote by $\cc^b_{ac}(\ca)$ the category of bounded acyclic complexes over $\ca$. And set $S$ to be the subset of $\ch(\cc^b(\ca))/I$ formed by all $r[K]$, where $r\in \mathbb{Q}^\times, K\in\cc_{ac}^b(\ca)$. The {\it modified Ringel-Hall algebra} $\cm\ch(\ca)$ is defined to be the right localization of $\ch(\cc^b(\ca))/I$ with respect to $S$, {\it i.e.}  $\cm\ch(\ca)=(\ch(\cc^b(\ca))/I)[S^{-1}]$ and we refer to ~\cite{LinP} for details. Here we also denote by $\diamond$ the multiplication in $(\ch(\cc^b(\ca)/I)[S^{-1}]$.

Given an object $A\in\ca$ and $m\in\mathbb{Z}$, let $K_{A,m}$ be the acyclic complex $$\cdots \rightarrow0 \rightarrow A \xrightarrow{1_A} A \rightarrow 0\rightarrow\cdots,$$ where $A$ sits in the degrees $m-1$ and $m$; $U_{A,m}$ be the stalk complex with $A$ concentrated in the degree $m$. By abuse of notations, we use the same symbols to denote their isomorphism classes in the modified Ringel-Hall algebra, {\it i.e.} $K_{A,m}:=[K_{A,m}]$ and $U_{A,m}:=[U_{A,m}]$. It is well-defined that $K_{\alpha, n}:=\frac{1}{\langle\alpha, \widehat{A_2}\rangle} K_{A_1, n}\diamond K^{-1}_{A_2, n}$, if $\alpha=\widehat{A_1}-\widehat{A_2}\in K_{0}(\ca)$ for two objects $A_1, A_2\in\ca$.
A $\mathbb{Q}$-basis of the modified Ringel-Hall algebra $\cm\ch(\ca)$ has been constructed in \cite{LinP}.
\begin{proposition}[\cite{LinP}]\label{proposition basis of modified}
$\cm\ch(\ca)$ has a basis consisting of elements
$$K_{\alpha_{r-1},r} \diamond K_{\alpha_{r-2},r-1}
\diamond\cdots\diamond K_{\alpha_{l+1},l+2} \diamond K_{\alpha_{l},l+1}
\diamond U_{A_r,r} \diamond U_{A_{r-1},r-1} \diamond\cdots\diamond U_{A_{l+1},l+1} \diamond U_{A_{l},l} ,$$
where $r\geq l$, $\alpha _i\in K_0(\ca)$ and $A_{j}\in\Iso(\ca)$ for any $l\leq i\leq r-1$ and $l\leq j\leq r$.
\end{proposition}

Let $A, B, M$ and $N$ be objects in $\ca$, we define a rational number $\gamma_{AB}^{MN}$. Let  $V(M, B, A, N)$ be the subset of $\Hom(M,B)\times\Hom(B,A)\times\Hom(A,N)$ consisting of exact sequences $0\rightarrow M \rightarrow B\rightarrow A\rightarrow N\rightarrow 0$. The set $V(M, B, A, N)$ is finite and let $\gamma_{AB}^{MN}:=\frac{|V(M, B, A, N)|}{a_Aa_B}$, here $a_X=|\aut_{\ca}(X)|$ denotes the cardinality of the automorphism group $\aut_{\ca}(X)$ for any object $X\in\ca$.

The modified Ringel-Hall algebra can be described by the generators and relations.

\begin{proposition}[\cite{LinP}]\label{proposition modified hall algebras}
The modified Ringel-Hall algebra $\cm\ch(\ca)$ is isomorphic to the associative and unital algebra generated by the set
$$\{U_{A,n},K_{\alpha,n}|A\in\Iso(\ca), \alpha\in K_0(\ca), n\in\mathbb{Z}\},$$  and the multiplication of which denoted by $\diamond$ subjects to the relations (\ref{relation in modified 1})-(\ref{relation in modified 10}) as follows.
\begin{eqnarray}
U_{A,n}\diamond U_{B,n}&=&\sum\limits_{C\in\Iso(\ca)}\frac{|\Ext^1_{\ca}(A,B)_C|}{|\Hom_{\ca}(A,B)|}U_{C,n},\label{relation in modified 1}\\
K_{\alpha,n}\diamond U_{A,n}&=&\langle\widehat{A},\alpha\rangle U_{A,n}\diamond K_{\alpha,n},\\
K_{\alpha, n}\diamond K_{\beta, n}&=&\frac{1}{\langle\alpha, \beta\rangle} K_{\alpha+\beta, n},
\end{eqnarray}
\begin{eqnarray}
U_{A,n}\diamond K_{\alpha,n+1}&=&\langle\alpha, \widehat{A}\rangle K_{\alpha,n+1}\diamond U_{A,n},\\
K_{\alpha,n}\diamond U_{A,n+1}&=&U_{A,n+1}\diamond K_{\alpha,n},\\
K_{\alpha,n}\diamond K_{\beta,n+1}&=&\langle\beta, \alpha\rangle K_{\beta,n+1}\diamond K_{\alpha,n},
\end{eqnarray}
\begin{eqnarray}
&U_{B,n}\diamond U_{A,n+1}=\sum\limits_{M,N\in\Iso(\ca)}\gamma_{AB}^{MN}\frac{a_Aa_B}{a_Ma_N}\langle\widehat{B}-\widehat{M}, \widehat{M}\rangle K_{\widehat{B}-\widehat{M},n+1}\diamond U_{N,n+1}\diamond U_{M,n},
\end{eqnarray}
and if $|m-n|\geq 2$, then
\begin{eqnarray}
U_{A,m}\diamond U_{B,n} &=& U_{B,n}\diamond U_{A,m},\\
K_{\alpha,m}\diamond U_{B,n} &=&U_{B,n}\diamond K_{\alpha,m},\\
K_{\alpha,m}\diamond K_{\beta,n}&=&K_{\beta,n}\diamond K_{\alpha,m}\label{relation in modified 10}.
\end{eqnarray}
\end{proposition}

\section{Modified Ringel-Hall Algebras and Naive Lattice Algebras}\label{section modified and naive}
For each finitary hereditary abelian $k$-category $\mathcal{A}$, Kapranov~\cite{Kap} constructed the so-called naive lattice algebra $\mathcal{N}(\mathcal{A})$ for $\mathcal{A}$.
The main purpose of this section is to compare the naive lattice algebra $\mathcal{N}(\mathcal{A})$ with the modified Ringel-Hall algebra $\cm\ch(\ca)$. We first recall the definition of the naive lattice algebras.
\subsection{Naive lattice algebras}
Given two Hopf algebras $\Xi$ and $\Omega$, a Hopf pairing is a bilinear map $\phi: \Xi\times\Omega\rightarrow\mathbb{}\mathbb{Q}(v)$ satisfying the following conditions:
\begin{itemize}
\item[(1)] $\phi(\xi,1)=\epsilon_\Xi(\xi), \phi(1,\omega)=\epsilon_\Omega(\omega),$
\item[(2)] $\phi(\xi\xi', \omega)=\phi^{\otimes 2}(\xi\otimes\xi',\Delta(\omega)),$
\item[(3)] $\phi(\xi, \omega\omega')=\phi^{\otimes 2}(\Delta(\xi),\omega\otimes\omega'),$
\end{itemize}
where $\Delta$ and $\epsilon$ denote the comultiplication and counit respectively and $$\phi^{\otimes 2}:(\Xi\otimes\Xi)\times(\Omega\otimes\Omega)\rightarrow \mathbb{Q}(v)$$ is the pairing defined by $(\xi\otimes\xi', \omega\otimes\omega')\longmapsto \phi(\xi,\omega)\phi(\xi',\omega')$. We do not include here any condition on the antipodes since we will not need them.

\begin{definition}[\cite{Kap}]\label{definition of naive lattice algebra}
Let $\Xi_{m}, m\in\mathbb{Z}$ be Hopf algebras and $\phi_m: \Xi_{m+1}\times\Xi_{m}\rightarrow \mathbb{Q}(v)$ be Hopf pairings. The naive lattice algebra $\cn=\cn(\{\Xi_{m}, \phi_m\})$ is generated by elements of all the algebras $\Xi_{m}$, so that inside each $\Xi_{m}$ the elements are multiplied according to the multiplication law there while for elements of different algebras we impose the relations
$$\xi_{m}\xi_{m+1}=(Id\otimes\phi_{m}\otimes Id)\left(\Delta_{\Xi_{m+1}}(\xi_{m+1})\otimes\Delta_{\Xi_{m}}(\xi_{m})\right)$$
and $$\xi_{m}\xi_{m'}=\xi_{m'}\xi_{m},\quad |m-m'|\geq 2.$$
\end{definition}

For any sequence $(a^i)_{i\in\mathbb{Z}}$ of elements of a possibly non-commutative algebra $S$, almost all equal to $1$, we define their ordered product to be $$\prod^{\overleftarrow{}}_ia^i=a^pa^{p-1}\cdots a^{q+1}a^q,$$where $p,q$ are integers such that $a^i=1$ unless $q\leq i\leq p$.

\begin{lemma}[\cite{Kap}]\label{lemma of iso between tensor algebras and naive algebras by Kap}
The ordered product map $\otimes_{m\in\mathbb{Z}}\Xi_{m}\rightarrow \cn$ is an isomorphism of vector spaces.
\end{lemma}

By the work of Green \cite{Gr} and Xiao \cite{X}, the comultiplication $\Delta$ and counit $\epsilon$ of twisted extended Ringel-Hall algebra $\ch^e_{tw}(\ca)$ are given by
$$\Delta:\ch^e_{tw}(\ca)\rightarrow \ch^e_{tw}(\ca)\widehat{\otimes} \ch^e_{tw}(\ca),\epsilon: \ch^e_{tw}(\ca)\rightarrow \mathbb{Q}(v),$$
\begin{eqnarray*}
\Delta([A]*k_\alpha)&=&\sum_{[B],[C]\in \Iso(\ca)}\sqrt{\langle \widehat{B},\widehat{C}\rangle}\frac{|\Ext^1_\ca(B,C)_A|}{|\Hom_{\ca}(B,C)|}\frac{a_A}{a_Ba_C}([B]*k_{\widehat{C}+\alpha})\otimes[C]*k_\alpha,\label{equation coproduct 1}\\
\epsilon([A]k_\alpha)&=&\delta_{[A],0}
\end{eqnarray*}
for any $[A]\in\Iso(\ca)$ and $\alpha\in K_0(\ca)$. Then $(\ch_{tw}^e(\ca),*,[0],\Delta,\epsilon)$ is a \emph{topological bialgebra} defined over $\mathbb{Q}(v)$.
Here topological means that everything should be considered in the completed space.

It is well-known that there exists a non-degenerate symmetric bilinear pairing $$\phi: \ch^e_{tw}(\ca)\widehat{\otimes} \ch^e_{tw}(\ca)\rightarrow\mathbb{Q}(v)$$
defined by $$\phi([M]*K_\alpha, [N]*K_\beta)=\sqrt{(\alpha,\beta)}\delta_{[M],[N]}a_M,$$ which is a Hopf pairing. Xiao \cite{X} showed that the extended twisted Ringel-Hall algebra has an antipode and it is a Hopf algebra.
The {\it naive lattice algebra } $\mathcal{N}(\mathcal{A})$ for the hereditary category $\mathcal{A}$ is the naive lattice algebra for $\Xi_{m}=\ch^e_{tw}(\ca)$ and $\phi_m=\phi$ for $m\in \mathbb{Z}$.

According to the construction, the naive lattice algebra $\cn(\ca)$ can also be described by the generators and relations. By Lemma \ref{lemma of iso between tensor algebras and naive algebras by Kap} and the proof of Proposition 1.5.3 of \cite{Kap}, one can easily get the following proposition.
\begin{proposition}\label{proposition of naive lattice algebra of hall algebra}
The naive lattice algebra $\cn(\ca)$ is generated by the symbols $Y_A^{(n)}$ and $ K_\alpha^{(n)}$ for all $A\in\Iso(\ca), n\in\mathbb{Z}$ and $\alpha\in K_{0}(\ca)$. And the following relations (\ref{relation in naive 1})-(\ref{relation in naive 10}) are the defining relations of $\cn(\ca)$.

\begin{eqnarray}
Y_A^{(n)}Y_B^{(n)}&=&\sum\limits_{C\in \Iso(\ca)}\sqrt{\langle \widehat{A},\widehat{B}\rangle}\frac{|\Ext^1_{\ca}(A,B)_C|}{|\Hom_{\ca}(A,B)|}Y_C^{(n)},\label{relation in naive 1}\\
K_\alpha^{(n)}Y_A^{(n)}&=&\sqrt{(\alpha, \widehat{A})}Y_A^{(n)}K_\alpha^{(n)},\\
K_\alpha^{(n)}K_\beta^{(n)}&=&K_{\alpha+\beta}^{(n)},
\end{eqnarray}
\begin{eqnarray}
Y_A^{(n)}K_\alpha^{(n+1)}&=&\sqrt{(\alpha, \widehat{A})}K_\alpha^{(n+1)}Y_A^{(n)},\label{relation of adjacent 1}\\
K_\alpha^{(n)}Y_A^{(n+1)}&=&Y_A^{(n+1)}K_\alpha^{(n)},\\
K_\alpha^{(n)}K_\beta^{(n+1)}&=&\sqrt{(\alpha, \beta)}K_\beta^{(n+1)}K_\alpha^{(n)},
\end{eqnarray}
\begin{eqnarray}
&Y_B^{(n)}Y_A^{(n+1)}=\sum\limits_{M,N\in\Iso(\mathcal{A})}\gamma_{AB}^{MN}\frac{a_Aa_B}{a_Ma_N}\sqrt{\langle\widehat{M}-\widehat{N}, \widehat{M}-\widehat{B}\rangle}Y_N^{(n+1)}Y_M^{(n)}K_{\widehat{B}-\widehat{M}}^{(n+1)}\label{relation of adjacent 4},
\end{eqnarray}
if $|m-n|\geq 2$, then
\begin{eqnarray}
Y_A^{(m)}Y_B^{(n)}&=&Y_B^{(n)}Y_A^{(m)},\\
K_{\alpha}^{(m)}Y_B^{(n)}&=&Y_B^{(n)}K_{\alpha}^{(m)},\\
K_{\alpha}^{(m)}K_{\beta}^{(n)}&=&K_{\beta}^{(n)}K_{\alpha}^{(m)}.\label{relation in naive 10}
\end{eqnarray}
\end{proposition}

\subsection{Componentwise Twisted Modified Ringel-Hall Algebras}

In the following, we define \emph{the componentwise Euler form} on $\Iso(\cc^b(\ca))$
$$\langle-,-\rangle_{cw}: \Iso(\cc^b(\ca))\times\Iso(\cc^b(\ca))\rightarrow \mathbb{Q}^\times,$$
by setting $\langle [M],[N]\rangle_{cw}=\sqrt{\prod\limits_{i\in\mathbb{Z}}\langle \widehat{M^i},\widehat{N^i}\rangle}$ for $[M], [N]\in \Iso(\cc^b(\ca))$.
This form descends to a bilinear form
$$\langle-,-\rangle_{cw}:K_0(\cc^b(\ca))\times K_0(\cc^b(\ca))\rightarrow\Q^\times.$$

The multiplication in the \emph{componentwise twisted modified Ringel-Hall algebra} $\cm\ch_{ctw}(\ca)$ is defined by
$$[M_1]\star[M_2]:=\langle [M_1],[M_2]\rangle_{cw}[M_1]\diamond[M_2],\quad \text{for any}~[M_1],[M_2]\in\Iso(\cc^{b}(\ca)).$$

So it is not hard to obtain that the \emph{componentwise twisted modified Ringel-Hall algebra} $\cm\ch_{ctw}(\ca)$ is also generated by
the set
$$\{U_{A,n},K_{\alpha,n}|A\in\Iso(\ca),\alpha\in K_0(\ca), n\in\mathbb{Z}\},$$
but subject to the relations (\ref{relation of same lattice in twisted modified 1})-(\ref{relation of not-adjacent in twisted modified 10}) in the following proposition.

\begin{proposition}\label{proposition twisted multiplication in modified hall algebras}
For any $A, B\in\Iso(\ca)$, $\alpha, \beta\in K_{0}(\ca)$, and $m, n\in\mathbb{Z}$, we have
\begin{eqnarray}
U_{A,n}\star U_{B,n}&=&\sum\limits_{C\in\Iso(\ca)}\sqrt{\langle \widehat{A},\widehat{B}\rangle}\frac{|\Ext^1_{\ca}(A,B)_C|}{|\Hom_{\ca}(A,B)|}U_{C,n},\label{relation of same lattice in twisted modified 1}\\
K_{\alpha,n}\star U_{A,n}&=&\sqrt{(\alpha, \widehat{A})}U_{A,n}\star K_{\alpha,n},\\
K_{\alpha, n}\star K_{\beta, n}&=&K_{\alpha+\beta, n},\label{relation of same lattice in twisted modified 3}
\end{eqnarray}
\begin{eqnarray}
U_{A,n}\star K_{\alpha,n+1}&=&\sqrt{(\alpha, \widehat{A})}K_{\alpha,n+1}\star U_{A,n},\label{relation in twisted modified 4 }\\
K_{\alpha,n}\star U_{A,n+1}&=&U_{A,n+1}\star K_{\alpha,n},\\
K_{\alpha,n}\star K_{\beta,n+1}&=&\sqrt{(\alpha, \beta)}K_{\beta,n+1}\star K_{\alpha,n},\label{relation in twisted modified 6}
\end{eqnarray}
\begin{eqnarray}
&U_{B,n}\star U_{A,n+1}=\sum\limits_{M,N\in\Iso(\ca)}\gamma_{AB}^{MN}\frac{a_Aa_B}{a_Ma_N}\sqrt{\langle\widehat{M}-\widehat{N}, \widehat{M}-\widehat{B}\rangle}U_{N,n+1}\star U_{M,n}\star K_{\widehat{B}-\widehat{M},n+1},\label{relation of not-adjacent in twisted modified 7}
\end{eqnarray}
if $|m-n|\geq 2$, then
\begin{eqnarray}
U_{A,m}\star U_{B,n} &=& U_{B,n}\star U_{A,m},\label{relation of not-adjacent in twisted modified 8}\\
K_{\alpha,m}\star U_{B,n} &=&U_{B,n}\star K_{\alpha,m},\\
K_{\alpha,m}\star K_{\beta,n}&=&K_{\beta,n}\star K_{\alpha,m}\label{relation of not-adjacent in twisted modified 10}
\end{eqnarray}
in $\cm\ch_{ctw}(\ca)$ , where $K_{\alpha, n}=K_{A_1, n}\star K^{-1}_{A_2, n}$, if $\alpha=\widehat{A_1}-\widehat{A_2}\in K_{0}(\ca)$.
\end{proposition}

\begin{proof}
Since the restriction of \emph{the componentwise twisted} multiplication on $\Iso(\ca)$ coincides with the twisted multiplication of $\ch^e_{tw}(\ca)$, it is trivial to check the relations (\ref{relation of same lattice in twisted modified 1})-(\ref{relation of same lattice in twisted modified 3}).  And it is not hard to check the relations (\ref{relation in twisted modified 4 })-(\ref{relation in twisted modified 6}) and  (\ref{relation of not-adjacent in twisted modified 8})-(\ref{relation of not-adjacent in twisted modified 10}) by the definition of \emph{the componentwise Euler form}.
For (\ref{relation of not-adjacent in twisted modified 7}), we have
\begin{eqnarray*}
U_{B,n}\star U_{A,n+1}&=&U_{B,n}\diamond U_{A,n+1}\\
&=&\sum\limits_{M,N\in \Iso(\ca)}\gamma_{AB}^{MN}\frac{a_Aa_B}{a_Ma_N}\langle\widehat{B}-\widehat{M}, \widehat{M}\rangle K_{\widehat{B}-\widehat{M},n+1}\diamond U_{N,n+1} \diamond U_{M,n}\\
&=&\sum\limits_{M,N\in \Iso(\ca)}\gamma_{AB}^{MN}\frac{a_Aa_B}{a_Ma_N}\frac{\langle \widehat{B}-\widehat{M},\widehat{M}\rangle\sqrt{(\widehat{B}-\widehat{M},\widehat{N})}}{\sqrt{\langle \widehat{B}-\widehat{M},\widehat{N}\rangle\langle \widehat{B}-\widehat{M},\widehat{M}\rangle}\sqrt{(\widehat{B}-\widehat{M},\widehat{M})}}\\
&&U_{N,n+1}\star U_{M,n}\star K_{\widehat{B}-\widehat{M},n+1}\\
&=&\sum\limits_{M,N\in \Iso(\ca)}\gamma_{AB}^{MN}\frac{a_Aa_B}{a_Ma_N}\sqrt{\langle\widehat{M}-\widehat{N}, \widehat{M}-\widehat{B}\rangle} U_{N,n+1}\star U_{M,n}\star K_{\widehat{B}-\widehat{M},n+1},
\end{eqnarray*}
which completes the proof.
\end{proof}

\begin{theorem}\label{theorem of isomorphism of naive lattice algebras and modified hall algebras}
The naive lattice algebra $\cn(\ca)$ is isomorphic to the componentwise twisted modified Ringel-Hall algebra $\cm\ch_{ctw}(\ca)$.
\end{theorem}
\begin{proof}
Following Proposition \ref{proposition of naive lattice algebra of hall algebra} and Proposition \ref{proposition twisted multiplication in modified hall algebras}, there is an isomorphism of algebras
$$\Theta:\cm\ch_{ctw}(\ca)\rightarrow\cn(\ca),$$
by setting
$$U_{A,n}\mapsto Y_A^{(n)}~\text{and}~K_{\alpha,n}\mapsto K_\alpha^{(n)}.$$
\end{proof}

\section{Modified Ringel-Hall Algebras and Lattice Algebras}\label{section modified and lattice}
\subsection{Derived Hall algebras}
Let $\ct$ be a $k$-additive triangulated category with translation $[1]$ satisfying
\begin{itemize}
\item[(i)] $\dim_k\Hom_{\ct}(X, Y)<\infty$ for any two objects $X$ and $Y$;
\item[(ii)] $\End_{\ct}(X)$ is local for any indecomposable object $X\in\ct$;
\item[(iii)] $\ct$ is (left) locally finite; that is, $\sum_{i\geq 0}\dim_k\Hom_{\ct}(X[i], Y)<\infty$ for any $X$ and $Y$ in $\ct$.
\end{itemize}

The {\it derived Hall algebra} $\cd\ch(\ct)$ of the triangulated category $\ct$ is the $\mathbb{Q}$-space with the basis $\{[X]|X\in\ct\}$ and the multiplication is defined by
$$[X][Y]=\sum_{[L]\in\Iso(\ct)}\frac{|\Ext^1_\ct(X, Y)_L|}{\prod_{i\geq 0}|\Hom_{\ct}(X[i], Y)|^{(-1)^i}}[L],$$
where $\Ext^1_\ct(X, Y)_L$ is defined to be $\Hom_{\ct}(X, Y[1])_{L[1]}$ which denotes the subset of $\Hom(X, Y[1])$ consisting of morphisms $l: X\rightarrow Y[1]$ whose cone $\Cone(l)$ is isomorphic to $L[1]$. Here we use the structure coefficient given by M. Kontsevich and Y. Soibelman in \cite{KS} which is different from the one introduced by B. To\"{e}n \cite{T} and Xiao-Xu \cite{XX1}, however it is proved in \cite{XX2} that both derived Hall algebras with these two different structure coefficients are isomorphic.

In particular, for the hereditary abelian category $\ca$ satisfying the finiteness conditions (1) and (2) in Section \ref{section introduction}, it is easy to describe the derived Hall algebra of $\ca$ by generators
$$\{Z_A^{[n]}|A\in\Iso(\ca), n\in\mathbb{Z}\},$$
and relations in terms of $\ca$, where $Z_A^{[n]}$ is the stalk complex with the non-zero component $A$ sitting in the degree $n$.

\begin{proposition}
$\cd\ch(\ca)$ is an associative and unital $\mathbb{Q}$-algebra generated by the set $$\{Z_A^{[n]}|A\in\Iso(\ca), n\in\mathbb{Z}\},$$ and subject to the following relations:
\begin{eqnarray}
Z_A^{[n]}Z_B^{[n]}&=&\sum\limits_{C\in \Iso(\ca)}\frac{|\Ext^1_\ca(A,B)_C|}{|\Hom_\ca(A,B)|}Z_C^{[n]},\\
Z_B^{[n]}Z_A^{[n+1]}&=&\sum\limits_{M, N\in\Iso(\ca)}\gamma_{AB}^{MN}\frac{a_Aa_B}{a_Ma_N}\frac{1}{\langle \widehat{N}, \widehat{M}\rangle}Z_N^{[n+1]}Z_M^{[n]},\\
Z_B^{[n]}Z_A^{[m]}&=&\langle \widehat{A}, \widehat{B}\rangle^{(-1)^{m-n}}Z_A^{[m]}Z_B^{[n]}\ \text{for~} m>n+1.
\end{eqnarray}
\end{proposition}

For any $[M], [N]\in\Iso(D^b(\ca))$, define the Euler form
$$\langle [M], [N]\rangle_t=\sqrt{\prod\limits_{i\in\mathbb{Z}}|\Hom_{D^b(\ca)}(M, N[i])|^{(-1)^i}},$$ also it can descend to the Grothendieck group $K_0(D^b(\ca))$. Moreover it coincides with the Euler form of $K_0(\ca)$ over the stalk complexes. And then the multiplication of the twisted derived Hall algebra $\cd\ch_{tw}(\ca)$ is given by
$$[M]\triangle[N]=\langle \widehat{M}, \widehat{N}\rangle_t[M][N],~\text{for any}~[M], [N]\in\Iso(D^b(\ca)).$$

\begin{definition}[\cite{SX}]\label{definition SX}
The extended twisted derived Hall algebra $\cd\ch^{e}_{tw}(\ca)$ is the associative algebra
generated by all the elements $$Z_A^{[n]}, K_\alpha,$$
for all $A\in\Iso(\ca), n\in\mathbb{Z}~\text{and}~\alpha\in K_{0}(\ca)$, and with the following defining relations.
\begin{eqnarray}
K_\alpha\triangle K_\beta=K_{\alpha+\beta},&&K_\alpha\triangle Z_A^{[n]}=\sqrt{(\widehat{A},\alpha)^{(-1)^n}}Z_A^{[n]}\triangle K_\alpha,\\
Z_A^{[n]}\triangle Z_B^{[n]}&=&\sum\limits_{C\in \Iso(\ca)}\sqrt{\langle \widehat{A}, \widehat{B}\rangle} \frac{|\Ext^1_\ca(A,B)_C|}{|\Hom_\ca(A,B)|}Z_C^{[n]},\label{relation twist tri ZnZn}\\
Z_B^{[n]}\triangle Z_A^{[n+1]}&=&\sum\limits_{M, N\in\Iso(\ca)}\gamma_{AB}^{MN}\frac{a_Aa_B}{a_Ma_N}\frac{1}{\sqrt{\langle \widehat{B}, \widehat{A}\rangle}\sqrt{\langle \widehat{N}, \widehat{M}\rangle}}Z_N^{[n+1]}\triangle Z_M^{[n]},\label{relation twist tri ZnZn+1}\\
Z_B^{[n]}\triangle Z_A^{[m]}&=&\sqrt{(\widehat{A} , \widehat{B})^{(-1)^{n-m}}}Z_A^{[m]}\triangle Z_B^{[n]}\ \text{for~} m>n+1.\label{relation twist tri ZnZm}
\end{eqnarray}
\end{definition}

\subsection{Lattice algebras}
For any objects $A, B, C\in\ca$, we use the symbol $g^C_{AB}$ to denote the number of subobject $B'$ of $C$ such that $B'\cong B$ and $C/B'\cong A$. Then one has the Riedtmann-Peng formula (see \cite{Rie,P2}) $$g^C_{AB}=\frac{|\Ext_{\ca}^1(A, B)_C|}{|\Hom_{\ca}(A, B)|}\frac{a_C}{a_Aa_B}.$$

\begin{definition}[\cite{Kap}]
The lattice algebra $\cl(\ca)$ is generated by the elements
$$X_A^{(n)}, K_\alpha,$$
for all $A\in\Iso(\ca), n\in\mathbb{Z}$ and $\alpha\in K_0(\ca)$, subject to the relations
\begin{eqnarray}
K_\alpha K_\beta=K_{\alpha+\beta}, &&K_\alpha X_A^{(n)}=\sqrt{(\widehat{A},\alpha)^{(-1)^n}}X_A^{(n)}K_\alpha,\\
X_A^{(n)}X_B^{(n)}&=&\sum\limits_{C\in \Iso(\ca)}\sqrt{\langle \widehat{A}, \widehat{B}\rangle}g_{AB}^{C}X_C^{(n)},\\
X_B^{(n)}X_A^{(n+1)}&=&\sum\limits_{M, N\in\Iso(\ca)}\gamma_{AB}^{MN}\sqrt{\langle \widehat{M}- \widehat{N},\widehat{M}-\widehat{B}\rangle}X_N^{(n+1)}X_M^{(n)}K_{\widehat{B}-\widehat{M}}^{(-1)^{n+1}},\\
X_B^{(n)}X_A^{(m)}&=&\sqrt{(\widehat{A} , \widehat{B})^{(-1)^{n-m}(n-m+1)}}X_A^{(m)}X_B^{(n)}\ \text{for~} |m-n|\geq 2
\end{eqnarray}
\end{definition}

By the result of Sheng and Xu \cite{SX}, we know that the lattice algebra coincides with the extended twisted derived Hall algebra $\cd\ch^e_{tw}(\ca)$ in Definition \ref{definition SX}.

Moreover, in terms of alternative generators $$Z_A^{(n)}=X_A^{(n)}a_{A}, K_\alpha,$$ for all $A\in\Iso(\ca)$, $n\in\mathbb{Z}$  and $\alpha\in K_0(\ca)$, one can easily get that the lattice algebra $\cl(\ca)$ is isomorphic to the algebra $\cl_*(\ca)$ which is described by the following proposition and  is called  the {\it Drinfeld dual lattice algebra} of $\ca$.

\begin{proposition}\label{proposition lattice algebra}
The lattice algebra $\cl(\ca)$ is isomorphic to the algebra $\cl_*(\ca)$ generated by the symbols $Z_A^{(n)}$ and $K_\alpha$, for all $A\in\Iso(\ca)$, $n\in\mathbb{Z}$ and $\alpha\in K_0(\ca)$, with defining relations as follows.
\begin{eqnarray}
K_\alpha K_\beta=K_{\alpha+\beta}, &&K_\alpha Z_A^{(n)}=\sqrt{(\widehat{A},\alpha)^{(-1)^n}}Z_A^{(n)}K_\alpha,\\
Z_A^{(n)}Z_B^{(n)}&=&\sum\limits_{C\in \Iso(\ca)}\sqrt{\langle \widehat{A}, \widehat{B}\rangle} \frac{|\Ext^1_\ca(A,B)_C|}{|\Hom_\ca(A,B)|}Z_C^{(n)},\\
Z_B^{(n)}Z_A^{(n+1)}&=&\sum\limits_{M, N\in\Iso(\ca)}\gamma_{AB}^{MN}\frac{a_Aa_B}{a_Ma_N}\sqrt{\langle \widehat{M}- \widehat{N},\widehat{M}-\widehat{B}\rangle}Z_N^{(n+1)}Z_M^{(n)}K_{\widehat{B}-\widehat{M}}^{(-1)^{n+1}},\\
Z_B^{(n)}Z_A^{(m)}&=&\sqrt{(\widehat{A} , \widehat{B})^{(-1)^{n-m}(n-m+1)}}Z_A^{(m)}Z_B^{(n)}\ \text{for~} |m-n|\geq 2.
\end{eqnarray}
\end{proposition}
By Proposition 3.3.2 of \cite{Kap}, one can obtain a $\mathbb{Q}$-basis of $\cl_*(\ca)$ as follows.

\begin{proposition}\label{proposition basis of lattice algebra }
The elements $K_\alpha Z_{A^r}^{(r)}Z_{A^{r-1}}^{(r-1)}\cdots Z_{A^l}^{(l)}$ forms a basis of $\cl_*(\ca)$, for all $A^i\in\Iso(\ca)(r, l\in\mathbb{Z}, l\leq i\leq r)$ and $\alpha\in K_0(\ca)$.
\end{proposition}

\subsection{Relative twisted modified Ringel-Hall algebras}In \cite{LinP}, we define the twisted modified Ringel-Hall algebra $
\cm\ch_{tw}(\ca)$ by the Euler form for $\Iso(\cc^b(\ca))$, {\it i.e.}
$$[M]*[N]=\langle [M],[N]\rangle[M]\diamond[N]$$
for any $[M],[N]\in\Iso(\cc^{b}(\ca))$, where $\langle [M],[N]\rangle=\prod_{p=0}^{+\infty} |\Ext^p_{\cc^b(\ca)}(M, N)|^{(-1)^p}$. In particular, the quantum torus of acyclic complexes is commutative with all the elements of $\cm\ch_{tw}(\ca)$ and we have the following result.
\begin{proposition}[\cite{LinP}]\label{proposition of relations in twist MH(A)}
$\cm\ch_{tw}(\ca)$ is generated by the set
$$\{U_{A,n},K_{\alpha,n}|A\in\Iso(\ca), \alpha\in K_0(\ca), n\in\mathbb{Z}\}$$ with defining relations as follows.
\begin{eqnarray}
U_{A,n}*U_{B,n}&=&\sum\limits_{C\in \Iso(\ca)}\langle \widehat{A},\widehat{B}\rangle\frac{|\Ext^1_{\ca}(A,B)_C|}{|\Hom_{\ca}(A,B)|}U_{C,n}, \\
K_{\alpha,n}*U_{A,n}&=&U_{A,n}*K_{\alpha,n},\\
K_{\alpha, n}*K_{\beta, n}&=&K_{\alpha+\beta, n},
\end{eqnarray}
\begin{eqnarray}
U_{A,n}*K_{\alpha,n+1}&=&K_{\alpha,n+1}*U_{A,n},\\
K_{\alpha,n}*U_{A,n+1}&=&U_{A,n+1}*K_{\alpha,n},\\
K_{\alpha,n}*K_{\beta,n+1}&=&K_{\beta,n+1}*K_{\alpha,n},
\end{eqnarray}
\begin{eqnarray}
\quad U_{B,n}*U_{A,n+1}&=\sum\limits_{M,N\in\Iso(\mathcal{A})}\gamma_{AB}^{MN}\frac{a_Aa_B}{a_Ma_N}\frac{1}{\langle\widehat{B}, \widehat{A}\rangle}U_{N,n+1}*U_{M,n}*K_{\widehat{B}-\widehat{M},n+1},
\end{eqnarray}
and if $m>n+1$, then
\begin{eqnarray}
U_{B,n}*U_{A,m}&=&\langle \widehat{B} , \widehat{A}\rangle^{(-1)^{m-n}}U_{A,m}*U_{B,n},\label{the quasi-commutative of twisted modified hall}\\
U_{B,n}*K_{\alpha,m}=K_{\alpha,m}*U_{B,n},&&U_{B,m}*K_{\alpha,n}=K_{\alpha,n}*U_{B,m}\\
K_{\beta,n}*K_{\alpha,m}&=&K_{\alpha,m}*K_{\beta,n},
\end{eqnarray}
for any $A, B\in\Iso(\ca)$, $\alpha, \beta\in K_{0}(\ca)$ and $m, n\in\mathbb{Z}$. And $K_{\alpha, n}=K_{A_1, n}* K^{-1}_{A_2, n}$, if $\alpha=\widehat{A_1}-\widehat{A_2}\in K_{0}(\ca)$.
\end{proposition}
\begin{definition}\label{definition relative Euler form}
For any $[M],[N]\in\Iso(C^b(\ca))$, define the relative Euler form for $\Iso(C^b(\ca))$ by setting
$$\langle [M],[N]\rangle_{r}=\sqrt{\prod\limits_{i,j\in\mathbb{Z}}\langle \widehat{M^i},\widehat{N^j}\rangle^{(-1)^{j-i+1}(j-i+1)}},$$
 and it is clear that it can descend to $K_{0}(C^b(\ca))$.
\end{definition}

\begin{proposition}\label{proposition relative Euler form}
For any $A, B\in\Iso(\ca)$, $\alpha, \beta\in K_0(\ca)$ and $n, m\in\mathbb{Z}$, we have
\begin{eqnarray}
\langle U_{A,n}, U_{B,m}\rangle_{r}&=&\sqrt{\langle \widehat{A},\widehat{B}\rangle^{(-1)^{m-n+1}(m-n+1)}},\label{relative euler form 1}\\
\langle K_{\alpha,n}, U_{B,m}\rangle_{r}&=&\sqrt{\langle \alpha,\widehat{B}\rangle^{(-1)^{m-n}}},\label{relative euler form 2}\\
\langle U_{B,n}, K_{\alpha,m}\rangle_{r}&=&\sqrt{\langle \widehat{B},\alpha\rangle^{(-1)^{m-n+1}}},\label{relative euler form 3}\\
\langle K_{\alpha,n}, K_{\beta,m}\rangle_{r}&=&1.\label{relative euler form 4}
\end{eqnarray}
\end{proposition}
\begin{proof}
The identity (\ref{relative euler form 1}) is induced by the Definition \ref{definition relative Euler form}.

For the identity (\ref{relative euler form 2}), if $\alpha=\widehat{A}\in K_0(\ca)$ for an object $A\in\ca$, then
$$\langle K_{\widehat{A},n}, U_{B,m}\rangle_{r}=\langle U_{A,n}, U_{B,m}\rangle_{r}\langle U_{A,n-1}, U_{B,m}\rangle_{r}=\sqrt{\langle \widehat{A},\widehat{B}\rangle^{(-1)^{m-n}}}.$$
If $\alpha=\widehat{A_1}-\widehat{A_2}\in K_0(\ca)$ for objects $A_1, A_2\in\ca$, then $K_{\alpha,n}=K_{\widehat{A_1},n}*K^{-1}_{\widehat{A_2},n}$. So $$\langle K_{\alpha,n}, U_{B,m}\rangle_{r}=\langle K_{\widehat{A_1},n}, U_{B,m}\rangle_{r}\langle K^{-1}_{\widehat{A_2},n}, U_{B,m}\rangle_{r}=\sqrt{\langle \alpha,\widehat{B}\rangle^{(-1)^{m-n}}}.$$ Similarly, one can get the identity (\ref{relative euler form 3}).

For the identity (\ref{relative euler form 4}), it suffices to prove
$$\langle K_{\widehat{A},n}, K_{\widehat{B},m}\rangle_{r}=1,$$
which is deduced from the identities $\langle K_{\widehat{A},n}, K_{\widehat{B},m}\rangle_{r}=\langle K_{\widehat{A},n}, U_{B,m}\rangle_{r}\langle K_{\widehat{A},n}, U_{B,m-1}\rangle_{r}$ and (\ref{relative euler form 2}).
\end{proof}

Let $\cm\ch_{rtw}(\ca)$ be the relative twisted modified Ringel-Hall algebra, with the relative twisted multiplication defined by
$$[M]\circ[N]=\langle [M],[N]\rangle_r[M]*[N],$$
{\it i.e.} $[M]\circ[N]=\langle [M],[N]\rangle_r \langle [M],[N]\rangle [M]\diamond[N]$ for any $[M],[N]\in\Iso(\cc^{b}(\ca))$.
Similarly, it is easy to get the following description of the relative twisted modified Ringel-Hall algebras by the generators and relations.

\begin{proposition}\label{proposition relative twisted modified}
The relative twisted modified Ringel-Hall algebra $\cm\ch_{rtw}(\ca)$ is isomorphic to the associative and unital algebra generated by the set
$$\{U_{A,n},K_{\alpha,n}|A\in\Iso(\ca), \alpha\in K_0(\ca), n\in\mathbb{Z}\}$$ and subject to the following relations (\ref{relation in relative modified 1})-(\ref{relation in relative modified 5}).
\begin{eqnarray}
K_{\alpha,n}\circ K_{\beta,m}=K_{\beta,m}\circ K_{\alpha,n},&& K_{\alpha,n}\circ K_{\beta,n}=K_{\alpha+\beta,n}\label{relation in relative modified 1}\\
K_{\alpha,m}\circ U_{A,n}&=&\sqrt{(\alpha,\widehat{A})^{(-1)^{n-m}}}U_{A,n}\circ K_{\alpha,m},\label{relation in relative modified 2}\\
U_{A,n}\circ U_{B,n}&=&\sum\limits_{C\in\Iso(\ca)}\sqrt{\langle \widehat{A},\widehat{B}\rangle}\frac{|\Ext^1_{\ca}(A,B)_C|}{|\Hom_{\ca}(A,B)|}U_{C,n},\label{relation in relative modified 3}
\end{eqnarray}
\begin{eqnarray}
U_{B,n}\circ U_{A,n+1}&=\sum\limits_{M,N\in\Iso(\ca)}\gamma_{AB}^{MN}\frac{a_Aa_B}{a_Ma_N}\sqrt{\langle\widehat{M}-\widehat{N},\widehat{M}-\widehat{B} \rangle}\label{relation in relative modified 4}\\&U_{N,n+1}\circ U_{M,n}\circ K_{\widehat{B}-\widehat{M},n+1},\nonumber
\end{eqnarray}
\begin{eqnarray}
U_{B,n}\circ U_{A,m}&=&\sqrt{(\widehat{A},\widehat{B})^{(-1)^{n-m}(n-m+1)}}U_{A,m}\circ U_{B,n},~ m-n\geq 2,
\label{relation in relative modified 5}
\end{eqnarray}
 for any $A, B\in\Iso(\ca)$, $\alpha, \beta\in K_{0}(\ca)$ and $m, n\in\mathbb{Z}$. And $K_{\alpha, n}=K_{\widehat{A_1}, n}\circ K^{-1}_{\widehat{A_2}, n}$, if $\alpha=\widehat{A_1}-\widehat{A_2}\in K_{0}(\ca)$.
\end{proposition}
\begin{proof}
The relation (\ref{relation in relative modified 1}) is a consequence of the commutativity of quantum torus of acyclic complexes in the twisted modified Ringel-Hall algebra and the identity (\ref{relative euler form 4}).
The relation (\ref{relation in relative modified 3}) is a direct consequence of Definition \ref{definition relative Euler form}.

For the relation (\ref{relation in relative modified 2}), we have
\begin{eqnarray*}
K_{\alpha,m}\circ U_{A,n}&=&\langle K_{\alpha,m}, U_{A,n}\rangle_r K_{\alpha,m}*U_{A,n}\\
&=&\sqrt{\langle \alpha, \widehat{A}\rangle^{(-1)^{n-m}}} U_{A,n}*K_{\alpha,m}\\
&=&\sqrt{\langle \alpha, \widehat{A}\rangle^{(-1)^{n-m}}}\frac{1}{\langle U_{A,n},  K_{\alpha,m}\rangle_r}U_{A,n}\circ K_{\alpha,m}\\
&=&\sqrt{(\alpha, \widehat{A})^{(-1)^{n-m}}}U_{A,n}\circ K_{\alpha,m}
\end{eqnarray*}
for any $m, n\in\mathbb{Z}$.

And for the relation (\ref{relation in relative modified 4}), we have
\begin{eqnarray*}
&&U_{B,n}\circ U_{A,n+1}\\
&=&\langle U_{B,n}, U_{A,n+1}\rangle_{r}U_{B,n}* U_{A,n+1}\\
&=&\sum\limits_{M,N\in\Iso(\ca)}\gamma_{AB}^{MN}\frac{a_Aa_B}{a_Ma_N} U_{N,n+1}*U_{M,n}*K_{\widehat{B}-\widehat{M},n+1}\\
&=&\sum\limits_{M,N\in\Iso(\ca)}\gamma_{AB}^{MN}\frac{a_Aa_B}{a_Ma_N}\frac{1}{\langle U_{N,n+1}, U_{M,n}\rangle_r\langle U_{M,n},K_{\widehat{B}-\widehat{M},n+1}\rangle_r\langle U_{N,n+1},K_{\widehat{B}-\widehat{M},n+1}\rangle_r}\\
&&U_{N,n+1}\circ U_{M,n}\circ K_{\widehat{B}-\widehat{M},n+1}\\
&=&\sum\limits_{M,N\in\Iso(\ca)}\gamma_{AB}^{MN}\frac{a_Aa_B}{a_Ma_N}
\sqrt{\langle\widehat{M}-\widehat{N},\widehat{M}-\widehat{B} \rangle}U_{N,n+1}\circ U_{M,n}\circ K_{\widehat{B}-\widehat{M},n+1}.
\end{eqnarray*}

For any $m\geq n+2$, we have
\begin{eqnarray*}
&&U_{B,n}\circ U_{A,m}\\
&=&\langle U_{B,n}, U_{A,m}\rangle_r U_{B,n}*U_{A,m}\\
&=&\sqrt{\langle \widehat{B},\widehat{A}\rangle^{(-1)^{m-n+1}(m-n+1)}}\langle \widehat{B} , \widehat{A}\rangle^{(-1)^{m-n}}U_{A,m}*U_{B,n}\\
&=&\frac{\sqrt{\langle \widehat{B},\widehat{A}\rangle^{(-1)^{m-n+1}(m-n+1)}}\langle \widehat{B} , \widehat{A}\rangle^{(-1)^{m-n}}}{\sqrt{\langle \widehat{A} , \widehat{B}\rangle^{(-1)^{n-m+1}(n-m+1)}}}U_{A,m}\circ U_{B,n}\\
&=&\sqrt{(\widehat{A},\widehat{B})^{(-1)^{n-m}(n-m+1)}}U_{A,m}\circ U_{B,n}.
\end{eqnarray*}
\end{proof}

\begin{theorem}\label{theorem epimorphism from modified to lattice}
(1) There is an epimorphism $\varphi:\cm\ch_{rtw}(\ca)\rightarrow\cl_*(\ca)$, given by $$K_{\alpha, n}\mapsto K^{(-1)^n}_{\alpha}, U_{A,n}\mapsto Z^{(n)}_{A}$$ for any $A\in\Iso(\ca)$, $n\in\mathbb{Z}$ and $\alpha\in K_0(\ca)$.

(2) $\cl_*(\ca)\cong \cm\ch_{rtw}(\ca)/I$, where $I$ is the ideal of $\cm\ch_{rtw}(\ca)$ generated by the set $$\{K_{\alpha, n+1}K_{\beta, n}-K_{\alpha-\beta, n+1}|\alpha, \beta\in K_{0}(\ca), n\in\mathbb{Z}\}.$$
\end{theorem}
\begin{proof}
(1) We just need to prove that $\varphi$ is a homomorphism. According to Proposition \ref{proposition lattice algebra} and Proposition \ref{proposition relative twisted modified}, it suffices to prove that the relations (\ref{relation in relative modified 1}) and (\ref{relation in relative modified 2}) are satisfied.

For any $\alpha, \beta\in K_0(\ca)$ and $n, m\in\mathbb{Z}$, we have
\begin{eqnarray*}
\varphi(K_{\alpha,n})\varphi(K_{\beta,m})&=&K^{(-1)^n}_{\alpha}K^{(-1)^m}_{\beta}\\
&=&K^{(-1)^m}_{\beta}K^{(-1)^n}_{\alpha}\\
&=&\varphi(K_{\beta,m})\varphi(K_{\alpha,n}).
\end{eqnarray*}
In particular, for $m=n$ we have
\begin{eqnarray*}
\varphi(K_{\alpha,n})\varphi(K_{\beta,n})&=&K^{(-1)^n}_{\alpha}K^{(-1)^n}_{\beta}\\
&=&K^{(-1)^n}_{\alpha+\beta}\\
&=&\varphi(K_{\alpha+\beta,n}).
\end{eqnarray*}

On the other hand, for any $A\in\Iso(\ca)$, we have
\begin{eqnarray*}
\varphi(K_{\alpha,m})\varphi(U_{A,n})&=&K_{\alpha}^{(-1)^m}Z_A^{(n)}\\
&=&\sqrt{((-1)^m\alpha, \widehat{A})^{(-1)^n}}Z_A^{(n)}K_{\alpha}^{(-1)^m}\\
&=&\sqrt{(\alpha,\widehat{A})^{(-1)^{n-m}}}Z_A^{(n)}K_{\alpha}^{(-1)^m}\\
&=&\varphi(U_{A,n})\varphi(K_{\alpha,m}).
\end{eqnarray*}

(2) It is equivalent to prove that $\Ker(\varphi)=I$, and it is obviously that $I\subseteq\Ker(\varphi)$.
One can easily see that the image of the basis described in Proposition \ref{proposition basis of modified} is the generators of the lattice algebra $\cl_*(\ca)$.

We claim that if $$K_{\alpha_{r-1},r}\circ\cdots\circ K_{\alpha_l,l+1}\circ U_{A^r,r}\circ\cdots\circ U_{A^l,l}-K_{\beta_{r-1},r}\circ\cdots\circ K_{\beta_l,l+1}\circ U_{A^r,r}\circ\cdots\circ U_{A^l,l}\in\Ker(\varphi)$$ for any two elements $$K_{\alpha_{r-1},r}\circ\cdots\circ K_{\alpha_l,l+1}\circ U_{A^r,r}\circ\cdots\circ U_{A^l,l}$$ and $$K_{\beta_{r-1},r}\circ\cdots\circ K_{\beta_l,l+1}\circ U_{A^r,r}\circ\cdots\circ U_{A^l,l}$$ in the basis of $\cm\ch_{rtw}(\ca)$, then $$K_{\alpha_{r-1},r}\circ\cdots\circ K_{\alpha_l,l+1}\circ U_{A^r,r}\circ\cdots\circ U_{A^l,l}-K_{\beta_{r-1},r}\circ\cdots\circ K_{\beta_l,l+1}\circ U_{A^r,r}\circ\cdots\circ U_{A^l,l}\in I.$$
It is clear that they must satisfy the following condition
$$(-1)^{r}\alpha_{r-1}+\cdots+(-1)^{l+1}\alpha_l=(-1)^{r}\beta_{r-1}+\cdots+(-1)^{l+1}\beta_l.$$
And it is induced by the following identity
\begin{eqnarray*}
&&K_{\alpha_{r-1},r}\circ\cdots\circ K_{\alpha_l,l+1}-K_{\beta_{r-1},r}\circ\cdots\circ K_{\beta_l,l+1}\\
&=&K_{\alpha_{r-1},r}K_{\alpha_{r-2},r-1}\cdots K_{\alpha_{l+2},l+3}(K_{\alpha_{l+1},l+2}K_{\alpha_{l},l+1}-K_{\alpha_{l+1}-\alpha_{l},l+2})\\
&-&K_{\beta_{r-1},r}K_{\beta_{r-2},r-1}\cdots K_{\beta_{l+2},l+3}(K_{\beta_{l+1},l+2}K_{\beta_{l},l+1}-K_{\beta_{l+1}-\beta_{l},l+2})\\
&+&K_{\alpha_{r-1},r}K_{\alpha_{r-2},r-1}\cdots K_{\alpha_{l+2},l+3}K_{\alpha_{l+1}-\alpha_{l},l+2}-K_{\beta_{r-1},r}K_{\beta_{r-2},r-1}\cdots K_{\beta_{l+2},l+3}K_{\beta_{l+1}-\beta_{l},l+2}
\end{eqnarray*}
that $$K_{\alpha_{r-1},r}K_{\alpha_{r-2},r-1}\cdots K_{\alpha_{l+2},l+3}K_{\alpha_{l+1}-\alpha_{l},l+2}-K_{\beta_{r-1},r}K_{\beta_{r-2},r-2}\cdots K_{\beta_{l+2},l+3}K_{\beta_{l+1}-\beta_{l},l+2}\in\Ker(\varphi).$$
Thus one can get that
$$K_{\alpha_{r-1},r}\circ\cdots\circ K_{\alpha_l,l+1}-K_{\beta_{r-1},r}\circ\cdots\circ K_{\beta_l,l+1}\in I$$
by induction on $r-l$.
Hence
$$K_{\alpha_{r-1},r}\circ\cdots\circ K_{\alpha_l,l+1}\circ U_{A^r,r}\circ\cdots\circ U_{A^l,l}-K_{\beta_{r-1},r}\circ\cdots\circ K_{\beta_l,l+1}\circ U_{A^r,r}\circ\cdots\circ U_{A^l,l}\in I.$$

For any
$$x=\sum\limits_{(\alpha_{r-1},\cdots,\alpha_l,A^r,\cdots,A^l)}a_{\alpha_{r-1},\cdots,\alpha_l,A^r,\cdots,A^l}K_{\alpha_{r-1},r}\circ\cdots\circ K_{\alpha_l,l+1}\circ U_{A^r,r}\circ\cdots\circ U_{A^l,l}\in\Ker(\varphi),$$
it is clear that $$x=\sum\limits_{(A^r,\cdots,A^l)}\left(\sum\limits_{(\alpha_{r-1},\cdots,\alpha_l)}a_{\alpha_{r-1},\cdots,\alpha_l,A^r,\cdots,A^l}K_{\alpha_{r-1},r}\circ\cdots\circ K_{\alpha_l,l+1}\right)\circ U_{A^r,r}\circ\cdots\circ U_{A^l,l}.$$

And it follows from the basis of $\cl_*(\ca)$ given in Proposition~\ref{proposition basis of lattice algebra } and the definition of~$\varphi$ that for any ~$A^r,\cdots,A^l\in\Iso(\ca)$ we can get that
$$\sum\limits_{(\alpha_{r-1},\cdots,\alpha_l)}a_{\alpha_{r-1},\cdots,\alpha_l,A^r,\cdots,A^l}K_{\alpha_{r-1},r}\circ\cdots\circ K_{\alpha_l,l+1}\in\Ker(\varphi).$$
In addition we have
\begin{eqnarray*}
&&\sum\limits_{(\alpha_{r-1},\cdots,\alpha_l)}a_{\alpha_{r-1},\cdots,\alpha_l,A^r,\cdots,A^l}K_{\alpha_{r-1},r}\circ\cdots\circ K_{\alpha_l,l+1}\\
&=&\sum\limits_{(\alpha_{r-1},\cdots,\alpha_l)}a_{\alpha_{r-1},\cdots,\alpha_l,A^r,\cdots,A^l}\left(K_{\alpha_{r-1},r}\circ\cdots\circ K_{\alpha_l,l+1}-K_{\alpha_{r-1},r}\circ\cdots\circ K_{\alpha_{l+1}-\alpha_l,l+2}\right)\\
&+&\sum\limits_{(\alpha_{r-1},\cdots,\alpha_l)}a_{\alpha_{r-1},\cdots,\alpha_l,A^r,\cdots,A^l}K_{\alpha_{r-1},r}\circ\cdots\circ K_{\alpha_{l+1}-\alpha_l,l+2},
\end{eqnarray*}
so by above claim and using induction on the number of $K_{\alpha_i}$, one can obtain that
$$\sum\limits_{(\alpha_{r-1},\cdots,\alpha_l)}a_{\alpha_{r-1},\cdots,\alpha_l,A^r,\cdots,A^l}K_{\alpha_{r-1},r}\circ\cdots\circ K_{\alpha_l,l+1}\in I,$$
which implies ~$\Ker(\varphi)\subseteq I$.
\end{proof}

\begin{remark}
It is inferred from Theorem \ref{theorem of isomorphism of naive lattice algebras and modified hall algebras} and Theorem \ref{theorem epimorphism from modified to lattice} that in some twisted case the lattice algebra is in fact the quotient algebra of the naive lattice algebra.
\end{remark}
\section{The invariance of naive lattice algebras}\label{the derived invariance}
In this section, we apply Theorem~\ref{theorem of isomorphism of naive lattice algebras and modified hall algebras} to prove that the naive lattice algebra is invariant under derived equivalences. In particular, it suffices to show that the componentwise twisted modified Ringel-Hall algebra is invariant under derived equivalences.

We begin with the following result concerning  derived Hall algebras.

\begin{proposition}\label{direct sum to multiplication}
For any object $[A]=[\bigoplus A^{-i}[i]]$ in $\Iso(D^b(\ca))$, we have
$$[A]=\sqrt{\prod_{i<j}\langle \widehat{A^j}, \widehat{A^i}\rangle^{(-1)^{j-i}}}\prod^{\leftarrow}_{n}Z_{A^{n}}^{[n]}$$ in $\cd\ch_{tw}(\ca)$, which is called the normal form of $[A]$  and  in the following we simply denote it by $Z(A)$.
\end{proposition}
\begin{proof}
For any object $[A]=[\bigoplus A^{-i}[i]$] of $\Iso(D^b(\ca))$. Assume that $A$ is of the following form $$\cdots \rightarrow0 \rightarrow A^l \xrightarrow{0} \cdots \xrightarrow{0} A^r\rightarrow 0\rightarrow\cdots,$$where $A^l$ is the leftmost nonzero component and $A^r$ is the rightmost nonzero component, then the \emph{width} of $A$ is defined to be $r-l+1$. If $A$ is zero, then the width of $A$ is defined to be $0$.  For simplicity set $l=0$, and inductively we have
\begin{eqnarray*}
\bigoplus^{r}_{i=0}Z_{A^i}^{[i]}&=&\frac{\prod\limits_{p\geq 0}|\Hom_{D^b(\ca)}(A^r[p],\bigoplus\limits^{r-1}_{i=0}A^i[r-i])|^{(-1)^p}}{\sqrt{\prod\limits_{0\leq i<r}\langle \widehat{A^r}, \widehat{A^i}\rangle^{(-1)^{r-i}}}}Z_{A^r}^{[r]} \vartriangle (\bigoplus\limits^{r-1}_{i=0}Z_{A^i}^{[i]})\\
 &=&\sqrt{\prod_{0\leq i<r}\langle \widehat{ A^r}, \widehat{A^i}\rangle^{(-1)^{r-i}}}Z_{A^r}^{[r]} \vartriangle  (\bigoplus\limits^{r-1}_{i=0}Z_{A^i}^{[i]})\\
  &=&\sqrt{\prod_{0\leq i<j\leq r}\langle \widehat{A^j}, \widehat{A^i}\rangle^{(-1)^{j-i}}}\prod^{\leftarrow}_{n}Z_{A^{n}}^{[n]}
\end{eqnarray*}
\end{proof}

It is clear that there is an automorphism $T$ of $\cd\ch_{tw}(\ca)$ by setting $T(Z_{A}^{[n]})=Z_A^{[n+1]}$.
\begin{remark}\label{remark derived invariance of derived hall algebras}
Let $\ca$ and $\cb$ be two hereditary abelian categories. If there exists a derived equivalence $F: D^b(\ca)\rightarrow D^b(\cb)$, then $\cd\ch_{tw}(\ca)$ is isomorphic to $\cd\ch_{tw}(\cb)$ by setting $$Z_A^{[n]}\mapsto T^{n}(Z(F(A))).$$
\end{remark}

By applying the method in~\cite{LinP}, it is not hard to see that
\begin{theorem}\label{Theorem embedding}
There is an embedding $\iota:\cd\ch_{tw}(\ca)\rightarrow \cm\ch_{ctw}(\ca)$, defined by
$$Z_A^{[0]}\mapsto U_{A,0},$$
$$Z_A^{[n]}\mapsto \frac{1}{\sqrt{\langle \widehat{A}, \widehat{A}\rangle^n}} U_{A,n}\star K_{-\widehat{A},n}\star K_{\widehat{A},n-1}\star \cdots \star K_{(-1)^{n-1}\widehat{A},2}\star K_{(-1)^{n}\widehat{A},1},$$
and $\ Z_A^{[-n]}\mapsto \sqrt{\langle \widehat{A}, \widehat{A}\rangle^n}U_{A,-n}\star K_{-\widehat{A},-n+1}\star K_{\widehat{A},-n+2}\star \cdots \star K_{(-1)^{n-1}\widehat{A},-1}\star K_{(-1)^{n}\widehat{A},0}$ for any $n>0$.
\end{theorem}

Following~ \cite{SX}, we define the \emph{completely extended twisted derived Hall algebra} $\cd\ch^{ce}_{tw}(\ca)$ of $\cd\ch_{tw}(\ca)$ and  show that the above embedding $\iota$ can be extended to an isomorphism between $\cd\ch^{ce}_{tw}(\ca)$ and $\cm\ch_{ctw}(\ca)$.

\begin{definition}\label{definition twisted extended derived Hall algebra}
The completely extended twisted derived Hall algebra $\cd\ch^{ce}_{tw}(\ca)$ is the associative and unital algebra generated by the set
$$\{Z_A^{[n]}, K_{\alpha}^{[n]}| A\in\Iso(\ca), \alpha\in K_0(\ca) \text{and}~n\in\mathbb{Z}\},$$
and subject to the relations (\ref{relation twist tri ZnZn})-(\ref{relation twist tri ZnZm}) and (\ref{relation in ex-de knZn})-(\ref{relation in ex-de knkm}).
\begin{eqnarray}
K_{\alpha}^{[n]} \vartriangle Z_A^{[n]}&=&\begin{cases}\sqrt{(\alpha, \widehat{A})}Z_A^{[n]} \vartriangle K_{\alpha}^{[n]}&n=0,1,\\Z_A^{[n]}\vartriangle K_{\alpha}^{[n]} &otherwise.\end{cases}\label{relation in ex-de knZn}\\
K_{\alpha}^{[n]} \vartriangle  K_{\beta}^{[n]}&=&K_{\alpha+\beta}^{[n]}\label{relation in ex-de knkn}\\
Z_A^{[n]} \vartriangle K_{\alpha}^{[n+1]}&=&\begin{cases}\sqrt{(\alpha, \widehat{A})}K_{\alpha}^{[n+1]} \vartriangle Z_A^{[n]}&n=0,-1,\\K_{\alpha}^{[n+1]}  \vartriangle Z_A^{[n]}&otherwise.\end{cases}\label{relation in ex-de Znkn+1}\\
K_{\alpha}^{[n]} \vartriangle Z_A^{[n+1]}&=&\begin{cases}\frac{1}{\sqrt{(\alpha, \widehat{A})}}Z_A^{[n+1]} \vartriangle K_{\alpha}^{[n]}&n=0,1,\\
Z_A^{[n+1]} \vartriangle K_{\alpha}^{[n]}&otherwise.\end{cases}\label{relation in ex-de knZn+1}\\
K_{\alpha}^{[n]} \vartriangle K_{\beta}^{[n+1]}&=&\sqrt{(\alpha, \beta)}K_{\beta}^{[n+1]} \vartriangle K_{\alpha}^{[n]}\label{relation in ex-de knKn+1}
\end{eqnarray}

If $|m-n|>1$, then
\begin{eqnarray}
K_{\alpha}^{[n]} \vartriangle Z_A^{[m]}&=&\begin{cases}\sqrt{(\alpha, \widehat{A})^{(-1)^m}}Z_A^{[m]} \vartriangle K_{\alpha}^{[0]}&n=0,|m|>1,\\
\frac{1}{\sqrt{(\alpha, \widehat{A})^{(-1)^m}}}Z_A^{[m]} \vartriangle K_{\alpha}^{[1]}&n=1,|m-1|>1,\\Z_{A}^{[m]} \vartriangle K_\alpha^{[n]}&otherwise;\end{cases}\label{relation in ex-de knZm}
\end{eqnarray}

\begin{eqnarray}
K_{\alpha}^{[n]} \vartriangle K_{\beta}^{[m]}&=&K_{\beta}^{[m]} \vartriangle K_{\alpha}^{[n]}.\label{relation in ex-de knkm}
\end{eqnarray}
\end{definition}

\begin{corollary}\label{proposition isomorphism of extended twisted derived and modified}
$\cd\ch^{ce}_{tw}(\ca)$ is isomorphic to $\cm\ch_{ctw}(\ca)$.
\end{corollary}
\begin{proof}Firstly, we can extend the embedding $\iota$ in Theorem \ref{Theorem embedding} to a morphism
$\widetilde{\iota}$ between $\cd\ch^{ce}_{tw}(\ca)$ and $\cm\ch_{ctw}(\ca)$ by setting
$$\widetilde{\iota}(Z_{A}^{[n]})=\iota(Z_{A}^{[n]})~\text{and}~\widetilde{\iota}(K_{\alpha}^{[n]})=K_{\alpha, n}.$$
Similarly, one can also check that $\widetilde{\iota}$~is a homomorphism of algebras.
To prove it is an isomorphism, we  construct the inverse homomorphism
$$\eta: \cm\ch_{ctw}(\ca)\rightarrow\cd\ch^{ce}_{tw}(\ca)$$
 given by
$$K_{\alpha,n}\mapsto K_{\alpha}^{[n]},$$
$$U_{A,0}\mapsto Z_{A}^{[0]},$$
and for $n>0$,
$$U_{A,n}\mapsto\sqrt{\langle\widehat{A}, \widehat{A}\rangle^{n}}Z_{A}^{[n]} \vartriangle K_{(-1)^{n+1}\widehat{A}}^{[1]} \vartriangle K_{(-1)^{n}\widehat{A}}^{[2]} \vartriangle \cdots \vartriangle K_{-\widehat{A}}^{[n-1]} \vartriangle K_{\widehat{A}}^{[n]},$$
$$U_{A,-n}\mapsto\frac{1}{\sqrt{\langle\widehat{A}, \widehat{A}\rangle^{n}}}Z_{A}^{[-n]} \vartriangle K_{(-1)^{n+1}\widehat{A}}^{[0]} \vartriangle K_{(-1)^{n}\widehat{A}}^{[-1]} \vartriangle \cdots \vartriangle K_{-\widehat{A}}^{[-n+2]} \vartriangle K_{\widehat{A}}^{[-n+1]}.$$
It is routine to check that
$$\eta\widetilde{\iota}=1_{\cd\ch^{ce}_{tw}(\ca)}~\text{and}~\widetilde{\iota}\eta=1_{\cm\ch_{ctw}(\ca)}.$$
\end{proof}

\begin{theorem}\label{theorem invariance of naive and lattice}
Let $\cb$ be also a hereditary abelian $k$-category satisfying the finiteness conditions (1)-(2) in Section \ref{section introduction}. If there exists a derived equivalence $F: D^b(\ca)\rightarrow D^b(\cb)$, then we have the following isomorphism of algebras
$$\cn(\ca)\cong \cn(\cb).$$
\end{theorem}
\begin{proof}
By Theorem \ref{theorem of isomorphism of naive lattice algebras and modified hall algebras} and Corollary \ref{proposition isomorphism of extended twisted derived and modified}, it suffices to prove that $\cd\ch^{ce}_{tw}(\ca)$ is isomorphic to $\cd\ch^{ce}_{tw}(\cb)$.

First of all, there is an isomorphism of Grothendieck group $$K_0(\ca)\rightarrow K_0(D^b(\ca))\rightarrow K_0(D^b(\cb))\rightarrow K_0(\cb)$$ induced by $F$. By abuse of notation this isomorphism is still denoted by $F$ and it preserves the bilinear form, {\it i.e.} for any $\alpha,\beta\in K_0(\ca)$ we have $$\langle\alpha,\beta\rangle_{\ca}=\langle F(\alpha),F(\beta)\rangle_{\cb},$$ where $\langle-,-\rangle_\ca$ and $\langle-,-\rangle_\cb$ denote the Euler form of $K_{0}(\ca)$ and $K_{0}(\cb)$ respectively.

Then the induced map $F_*:\cd\ch^{ce}_{tw}(\ca)\rightarrow\cd\ch^{ce}_{tw}(\cb)$ is given by
$$Z_{A}^{[n]}\mapsto T^n(Z(F(A)))~\text{and}~K_{\alpha}^{[n]}\mapsto K_{F_{(\alpha)}}^{[n]},$$
where $Z(F(A))$ is defined in Proposition \ref{direct sum to multiplication}.
It remains  to verify that it is a homomorphism of algebras. However, by Remark \ref{remark derived invariance of derived hall algebras} one can get that $F_*$ preserve the relations (\ref{relation twist tri ZnZn})-(\ref{relation twist tri ZnZm}). And it is easy to check that the relations (\ref{relation in ex-de knkn}), (\ref{relation in ex-de knKn+1}) and (\ref{relation in ex-de knkm}) are preserved.

Following~ \cite{Kap} and ~\cite{Cr}, let $\ca_i$  be the full subcategory of $\ca$ with objects $\{A\in\ca|F(A)\subset \cb[i]\}$. For objects $A_i\in\ca_i, A_j\in\ca_j$, we have $\Hom(A_i, A_j)=0$ but for $j=i$ or $j=i+1$, $\Ext^1_{\ca}(A_i, A_j)=0$ but for $j=i$ or $j=i-1$. Thus $$Z_{A_i\oplus A_j}^{[n]}=\sqrt{\langle\widehat{A_i}, \widehat{A_j}\rangle}Z_{A_i}^{[n]} \vartriangle Z_{A_j}^{[n]},$$ for $A_i\in\ca_i, A_j\in\ca_j~\text{and}~i<j.$
Since any object $A\in\ca$ can be decomposed to a direct sum $\bigoplus_{i\in \Z}A_i$ with $A_i\in\ca_i$, and then for any $n\in\mathbb{Z}$, $Z_A^{[n]}$ can be written in the form of
$$Z_A^{[n]}=\sqrt{\prod_{i<j}\langle\widehat{A_i}, \widehat{A_j}\rangle}\prod_{i\in S} Z_{A_i}^{[n]},$$ where the indices in the product $\prod\limits_{i\in S} Z_{A_i}^{[n]}$ are in increasing order. Therefore we only need to check the relations (\ref{relation in ex-de knZn}), (\ref{relation in ex-de Znkn+1}), (\ref{relation in ex-de knZn+1}) and (\ref{relation in ex-de knZm}) are preserved in the case that $A\in\ca_i$ and $\alpha\in K_0(\ca)$ for any $i\in\mathbb{Z}$.

Assume that $A\in\ca_{-i}$ and $F(A)=B[-i]$ for some object $B\in\cb$, then $Z(F(A))=Z_{B}^{[i]}$. In the following we give the proof for the relation~(\ref{relation in ex-de knZn}) and omit the others.
We separate the proof  into the following three cases.

Case (1): $n=0$.

If $i=-1$,
\begin{eqnarray*}
F_*(K_\alpha^{[0]}) \vartriangle F_*(Z_A^{[0]})&=&K_{F(\alpha)}^{[0]} \vartriangle Z_{B}^{[-1]}\\
&=&\frac{1}{\sqrt{(F(\alpha),\widehat{B})}}Z_{B}^{[-1]} \vartriangle K_{F(\alpha)}^{[0]}\\
&=&\frac{1}{\sqrt{(F(\alpha),F(A)[-1])}}Z_{B}^{[-1]} \vartriangle K_{F(\alpha)}^{[0]}\\
&=&\sqrt{(\alpha,\widehat{A})}F_*(Z_A^{[0]}) \vartriangle F_*(K_\alpha^{[0]}).
\end{eqnarray*}

If $i=0$,
\begin{eqnarray*}
F_*(K_\alpha^{[0]}) \vartriangle F_*(Z_A^{[0]})&=&K_{F(\alpha)}^{[0]} \vartriangle Z_{B}^{[0]}\\
&=&\sqrt{(F(\alpha),\widehat{B})}Z_{B}^{[0]} \vartriangle K_{F(\alpha)}^{[0]}\\
&=&\sqrt{(\alpha,\widehat{A})}F_*(Z_A^{[0]}) \vartriangle F_*(K_\alpha^{[0]}).
\end{eqnarray*}

If $i=1$,
\begin{eqnarray*}
F_*(K_\alpha^{[0]}) \vartriangle F_*(Z_A^{[0]})&=&K_{F(\alpha)}^{[0]} \vartriangle Z_{B}^{[1]}\\
&=&\frac{1}{\sqrt{(F(\alpha),\widehat{B})}}Z_{B}^{[1]} \vartriangle K_{F(\alpha)}^{[0]}\\
&=&\sqrt{(\alpha,\widehat{A})}F_*(Z_A^{[0]}) \vartriangle F_*(K_\alpha^{[0]}).
\end{eqnarray*}

If $i>1$ or $i<-1$,
\begin{eqnarray*}
F_*(K_\alpha^{[0]}) \vartriangle F_*(Z_A^{[0]})&=&K_{F(\alpha)}^{[0]} \vartriangle Z_{B}^{[i]}\\
&=&\sqrt{(F(\alpha),\widehat{B})^{(-1)^i}}Z_{B}^{[i]} \vartriangle K_{F(\alpha)}^{[0]}\\
&=&\sqrt{(\alpha,\widehat{A})}F_*(Z_A^{[0]}) \vartriangle F_*(K_\alpha^{[0]}).
\end{eqnarray*}

case (2): $n=1$.

If $i=-1$,
\begin{eqnarray*}
F_*(K_\alpha^{[1]}) \vartriangle F_*(Z_A^{[1]})&=&K_{F(\alpha)}^{[1]} \vartriangle Z_{B}^{[0]}\\
&=&\frac{1}{\sqrt{(F(\alpha),\widehat{B})}}Z_{B}^{[0]} \vartriangle K_{F(\alpha)}^{[1]}\\
&=&\sqrt{(\alpha,\widehat{A})}F_*(Z_A^{[1]}) \vartriangle F_*(K_\alpha^{[1]}).
\end{eqnarray*}

If $i=0$,
\begin{eqnarray*}
F_*(K_\alpha^{[1]}) \vartriangle F_*(Z_A^{[1]})&=&K_{F(\alpha)}^{[1]} \vartriangle Z_{B}^{[1]}\\
&=&\sqrt{(F(\alpha),\widehat{B})}Z_{B}^{[1]} \vartriangle K_{F(\alpha)}^{[1]}\\
&=&\sqrt{(\alpha,\widehat{A})}F_*(Z_A^{[1]}) \vartriangle F_*(K_\alpha^{[1]}).
\end{eqnarray*}

If $i=1$,
\begin{eqnarray*}
F_*(K_\alpha^{[1]}) \vartriangle F_*(Z_A^{[1]})&=&K_{F(\alpha)}^{[1]} \vartriangle Z_{B}^{[2]}\\
&=&\frac{1}{\sqrt{(F(\alpha),\widehat{B})}}Z_{B}^{[2]} \vartriangle K_{F(\alpha)}^{[1]}\\
&=&\sqrt{(\alpha,\widehat{A})}F_*(Z_A^{[1]}) \vartriangle F_*(K_\alpha^{[1]}).
\end{eqnarray*}

If $i>1$ or $i<-1$,
\begin{eqnarray*}
F_*(K_\alpha^{[1]}) \vartriangle F_*(Z_A^{[1]})&=&K_{F(\alpha)}^{[1]} \vartriangle Z_{B}^{[i+1]}\\
&=&\frac{1}{\sqrt{(F(\alpha),\widehat{B})^{(-1)^{i+1}}}}Z_{B}^{[i+1]} \vartriangle K_{F(\alpha)}^{[0]}\\
&=&\sqrt{(\alpha,\widehat{A})}F_*(Z_A^{[1]}) \vartriangle F_*(K_\alpha^{[1]}).
\end{eqnarray*}

Case (3): $n\neq 0, 1$.

\begin{eqnarray*}
F_*(K_\alpha^{[n]}) \vartriangle F_*(Z_A^{[n]})&=&K_{F(\alpha)}^{[n]} \vartriangle Z_{B}^{[n+i]}\\
&=&Z_{B}^{[n+i]} \vartriangle K_{F(\alpha)}^{[n]}\\
&=&F_*(Z_A^{[n]}) \vartriangle F_*(K_\alpha^{[n]}).
\end{eqnarray*}
This completes the proof.

\end{proof}


\begin{thebibliography}{AAA}
\bibitem{Br}Bridgeland T. Quantum groups via Hall algebras of complexes. Ann Math, 2013, 177: 739-759

\bibitem{Cr}Cramer T. Double Hall algebras and derived equivalences. Adv Math, 2010, 224: 1097-1120

\bibitem{Gor13}Gorsky M. Semi-derived Hall algebras and tilting invariance of Bridgeland-Hall algebras. ArXiv: 1303.5879, 2013

\bibitem{Gr}Green J. Hall algebras, hereditary algebras and quantum groups. Invent Math, 1995, 120: 361-377

\bibitem{Hub}Hubery A. From triangulated categories to Lie algebras: A theorem of Peng and Xiao, Trends in representation theory of algebras and related topics. Contemp Math, 2006, 406: 51-66

\bibitem{Kap}Kapranov M. Heisenberg doubles and derived categories. J Algebra, 1997, 202: 712-744

\bibitem{KS}Kontsevich M, Soibelman Y. Stability structures, motivic Donaldson-Thomas and cluster transformations. ArXiv:0811.2435, 2008

\bibitem{LinP}Lin J, Peng L G. Modified Ringel-Hall Algebras, Green's Formula and Derived Hall Algebras. ArXiv:1707.08292v2, 2017

\bibitem{LuP}Lu M, Peng L G. Modified Ringel-Hall algebras and drinfeld double. ArXiv:1608.03106, 2016

\bibitem{Lu}Lusztig G. Introduction to quantum groups. Boston: Birkh\"{a}user progress Math 110, 1993

\bibitem{P1}Peng L G. Some Hall polynomials for representation-finite trivial extention algebras. J Algebra, 1997, 197: 1-13

\bibitem{P2}Peng L G. Lie algebras determined by finite Auslander-Reiten quivers. Comm in Alg, 1999, 26: 2711-2725

\bibitem{PX1}Peng L G, Xiao J. Root categories and simple Lie algebras. J Algebra, 1997, 198: 19-56

\bibitem{PX2}Peng L G, Xiao J. Triangulated categories and Kac-Moody algebras. Invent Math, 2000, 140: 563-603

\bibitem{Rie} Riedtmann C. Lie algebras generated by indecomposables. J Algebra, 1994, 170: 526-546

\bibitem{R1}Ringel C. Hall algebras. In: Topics in Algebra, Banach Center Publ. 1990, 26: 433-447
\bibitem{R2}Ringel C. Hall algebras and quantum groups. Invent Math, 1990,101: 583-591
\bibitem{R3}Ringel C. Green's theorem on Hall algebras. In: Representation Theory of Algebras and Related Topics (Mexico
City, 1994), in: CMS Conference Proceedings, vol. 19. American Mathematical Socity, Providence, RI, 1996, 185-245

\bibitem{SV}Sevenhant B,Van den Bergh M. On the double of the Hall algebra of a quiver. J Algebra, 1999, 221: 135-160

\bibitem{SX}Sheng J, Xu F. Derived Hall algebras and lattice algebras. Algebra Colloq, 2012, 19: 533-538

\bibitem{T}To\"{e}n B. Derived Hall algebras. Duke Math J, 2006, 135: 587-615

\bibitem{X}Xiao J. Drinfeld double and Ringel-Green theory of Hall algebras. J Algebra, 1997, 190: 100-144

\bibitem{XX1}Xiao J, Xu F. Hall algebras associated to triangulated categories. Duke Math J, 2008, 143: 357-373
\bibitem{XX2}Xiao J, Xu F. Remarks on Hall algebras of triangulated categories. Kyoto J Math, 2015, 55: 477-499

\bibitem{Z}Zhang H C. Bridgeland's Hall algebras and Heisenberg doubles. J. Algebra Appl. 2018, 17(6): 12 pp

\end{thebibliography}
\end{document}